\documentclass[aop]{imsart}

\RequirePackage{amsthm,amsmath,amsfonts,amssymb}
\RequirePackage[authoryear]{natbib}
\RequirePackage[colorlinks,citecolor=blue,urlcolor=blue]{hyperref}
\RequirePackage{graphicx}
\usepackage{mathtools}

\startlocaldefs
\theoremstyle{plain}

\newtheorem{theorem}{Theorem}[section]
\newtheorem{lemma}{Lemma}[section]

\theoremstyle{remark}

\theoremstyle{remark}
\newtheorem{remark}{Remark}
\newtheorem{corollary}{Corollary}[section]
\newtheorem{assumption}{Assumption}[section]

\endlocaldefs

\newcommand{\eps}{\varepsilon}
\newcommand{\cA}{\mathcal{A}}

\newcommand{\RR}{\mathbb{R}}
\newcommand{\ZZ}{\mathbb{Z}}

\newtheorem{proposition}{Proposition}

\newcommand{\bbr}{\RR}
\newcommand{\bbz}{\ZZ}
\newcommand{\one}{{\bf 1}}
\newcommand{\Var}{{\rm Var}}
\newcommand{\vep}{\varepsilon}
\newcommand{\eid}{\stackrel{d}{=}}

\usepackage{etoolbox, xcolor}
\newtoggle{answers}
\newcommand{\interior}[1]{%
  {\kern0pt#1}^{\mathrm{o}}%
}

\togglefalse{answers}

\begin{document}

\begin{frontmatter}
\title{Clustering of large deviations events in heavy-tailed moving average processes: the catastrophe principle in the short-memory case}
\runtitle{Clustering of large deviations}

\begin{aug}
\author[A]{\fnms{Jiaqi}~\snm{Wang}  \ead[label=e1]{jw2382@cornell.edu}},
\author[B]{\fnms{Gennady}~\snm{Samorodnitsky}\ead[label=e2]{gs18@cornell.edu}}

\address[A]{School of ORIE, Cornell University \printead[presep={ ,\ }]{e1}}

\address[B]{School of ORIE, Cornell University \printead[presep={,\ }]{e2}}
\end{aug}

\begin{abstract}
How do large deviation events in a stationary process cluster? The
answer depends not only on the type of large deviations, but also on
the length of  memory in the process. Somewhat unexpectedly, it may
also depend on the tails of the process. In 
this paper we work in the context of large deviations for partial sums
in moving average processes with short memory and regularly varying
tails. We show that the structure of the large deviation cluster in
this case markedly differs from  the corresponding structure in the
case of exponentially light tails, considered in
\cite{chakrabarty:samorodnitsky:2024}. This is due to the difference
between the ``conspiracy'' vs. the ``catastrophe'' principles
underlying the large deviation events in the light tailed case and the
heavy tailed case, correspondingly. 
\end{abstract}



\end{frontmatter}

\section{Introduction}

Certain ``extreme'' events—such as hurricanes, earthquakes, power outages,
or financial crises—occur rarely but may have catastrophic
consequences. In certain systems, these rare events do not occur in
isolation; instead, they may cluster, with multiple rare events
happening close together in 
time. When such clustering occurs, the impact of individual events is
amplified, and this can potentially lead to prolonged disruptions and
cause systemic risks. Understanding how such 
clustering occurs is crucial for assessing risks and developing strategies to
mitigate their  effects.

Probabilistic modeling of such phenomena is naturally based on large
deviation events, the events that occur with vanishingly small
probability as the size of the system grows. 
Large deviation theory provides not only estimates for the 
probabilities of rare events, but it also describes the
``circumstances'' 
that are likely to lead to such events. These ``circumstances'' differ
if the underlying stochastic model has exponentially light tails, and
if it has heavier, subexponential, tails. We will see  
that this may lead to different clustering patterns of large
deviations events in the two cases.  

A bit more specifically (but still very informally), rare events in a 
stochastic model with exponentially light tails tend to exhibit the
``conspiracy principle'': these unlikely events are most likely to
happen because the random inputs driving the model all change their
behavior as if they are sampled from a ``wrong'' distribution - they
``conspire'' to do so. While each individual change may be mild, the
overall effect of the ``conspiracy'' is major. In contrast, rare
events in a 
stochastic model with subexponential tails tend to exhibit the
``catastrophe principle'': while the great majority of the random
inputs driving the model behave in their usual manner, a small number
of these inputs undergo a ``catastrophic'' change, sufficient to cause
the rare event. In many cases a single random input exhibits such
``catastrophic'' change!
We refer the reader to
\cite{dembo:zeitouni:1993} for an exposition of the large deviation
theory in the case of exponentially light tails and to
\cite{denisov:dieker:shneer:2008} for large deviations in the case of subexponential
tails.

 The specific behavior of these random inputs depends on the stochastic model,
and the precise effect of the ``conspiracy'' and ``catastrophe''
principles depends also on the memory in the model. The interplay
between the tails and the memory can be quite intricate.   

One of the main contributions of this paper is to demonstrate that
clustering of large deviation events may be very different if the
underlying mechanism is the ``catastrophe'' phenomenon as opposed to
the ``conspiracy'' phenomenon. Intuitively, the conspiracy phenomenon
is often reflected in a temporary change of the drift, and the rare events
keep happening for as long as that temporary change of the drift
suffices.  In contrast, under the catastrophe phenomenon, clustering
occurs when rare events share a common cause -- often a  single
extreme input-- making the location of the extreme input crucial. One
consequence of this is that in the regime considered in this paper
(the short memory regime),  the cluster of large deviations changes
from finite to infinite when the tails change from exponentially light
to power-like heavy.

The way rare events cluster can be expected to depend further on the
nature of the rare events. 
In this work, we consider finite mean doubly infinite moving average processes
with regularly varying tails, with (properly defined) short
memory. Whenever well-defined (convergent), these processes are
stationary and 
ergodic (even mixing) regardless of their tails. The large deviation
events we consider are the events on  which the sample averages of the
consecutive observations are away from the true mean. This is also the setup of
\cite{chakrabarty:samorodnitsky:2024}, except that the latter work
considered processes with exponentially light tails. Unlike that
paper that studied only univariate models,  our processes are allowed
to take valued in $\bbr^d$ with $d\geq 1$. 

Let
\begin{equation}\label{eq:def_moving_avg}
X_k = \sum_{i = -\infty}^\infty A_i Z_{k-i}, \ \ k\in\bbz, 
\end{equation}
where  $(Z_i)_{i\in \ZZ}$ is a sequence of i.i.d. $d$-dimensional
random vectors (the noise random vectors), and  $(A_i)_{i\in \ZZ}$ are
deterministic $d\times d$ matrices (the coefficient
matrices). In the sequel we will use the notation  $\|A\| = \|A\|_2$
to denote the $\ell_2/\ell_2$ operator norm of a matrix $A$, and all
norms of vectors are also the $\ell_2$ norms.

We assume
 that the noise vectors $(Z_i)$ have zero mean and a  regularly
 varying distribution with exponent $\alpha>1$. That is, there exists
 a Radon measure $\nu$ on $\bar{\RR}^d \setminus \{0\}$ such that  
\begin{equation}\label{eq:def_regular_varying_Rd}
    \frac{P(Z_0\in u\cdot)}{P(\|Z_0\| > u)}\xrightarrow{v} \nu(\cdot)
\end{equation}
(vague convergence) as $u\to \infty$, and  $\nu$ (the tail measure of
the noise) satisfies $\nu(c\Phi)
= c^{-\alpha} \nu (\Phi)$ for all measurable sets $\Phi \subset
\bar{\RR}^d \setminus \{0\}$ and $c>0$. The common notation is $Z_0\in RV(\nu, \alpha)$.

We assume that the coefficient matrices $(A_i)$ satisfy 
\begin{equation}\label{eq:def_coeff1}
    \sum_{i = -\infty}^{\infty} \| A_i \| \coloneqq A_{\text{abs}} <
    \infty \ \ \text{and the matrix} \ \ \mathcal{A} = \sum_{i =
      -\infty}^{\infty} A_i \ \ \text{is invertible.}
\end{equation}
The first condition in \eqref{eq:def_coeff1} is the absolute
summability (of the norms of the coefficients) assumption. It is
clearly sufficient for the a. s. convergence of the series \eqref{eq:def_moving_avg}
defining the infinite moving average process. Weaker summability
conditions on the coefficient norms are also sufficient, however; see
\cite{hult:samorodnitsky:2008}. The absolute summability condition is
usually taken as producing a short memory moving average process. The
second condition in \eqref{eq:def_coeff1}  avoids certain
cancellations corresponding to the null space of the matrix
$\cA$. Such cancellation, when present, can lead to certain negative
dependence phenomena that we wish to avoid in this paper.

Our rare events are defined by a  failure set $\Gamma\subset \bbr^d$, a measurable
set bounded away from the origin, satisfying certain regularity
assumptions; see Assumption \ref{ass:failure.set} below.  Denote 
\begin{equation} \label{e:rare.event}
B_j(n, \Gamma) = \left\{\frac{1}{n} \sum_{i = j}^{n+j - 1} X_i \in
  \Gamma \right\}, \ j\in\bbz.
\end{equation}
By stationarity, for each fixed $n$, the events $(B_j(n, \Gamma))$ have
the same probability and, when $n$ is large, this probability is small
by the law of large numbers. 

We are interested in the clustering of rare events: given that $B_0(n,
\Gamma)$ happens, how many nearby rare events $B_j (n, \Gamma)$ also
happen? In a marked difference from the case of exponentially light
tails considered in \cite{chakrabarty:samorodnitsky:2024}, given $B_0(n,
\Gamma)$, the conditional probability of $B_j(n,
\Gamma)$ converges to 1 as $n\to\infty$, for each fixed
$j$. Therefore, we choose to investigate the size of the cluster of
large deviation events by studying the first time these events do not
happen.  Specifically, let 
\begin{equation} \label{e:first.stop}
J_n^+ = \inf\{j>0: \, B_j(n, \Gamma) \, \text{ does not occur}\},  \ J _n^- = \sup\{j<0: \, B_j(n, \Gamma) \, \text{ does not occur}\}.
\end{equation}
We will investigate the joint limiting distribution of the properly normalized
$(J_n^+,J_n^-)$.

\begin{remark} \label{rk:extensions}
When the 
first condition in \eqref{eq:def_coeff1} fails, the moving average is
sometimes said to be long-range dependent. In that case the large
deviations are expected to cluster in a different manner. When the
tails are exponentially light, this was demonstrated in
\cite{chakrabarty:samorodnitsky:2023} in the one-dimensional
case. When the 
second condition in \eqref{eq:def_coeff1} fails, the moving average is
sometimes said to have negative long-range dependent. How the large
deviations cluster in this case is, to the best of our knowledge,
still unknown event when the tails are exponentially light. We plan to
address these questions for heavy tailed moving average processes in a
subsequent publication. 
\end{remark} 

The following notation and terminology will be used
throughout the 
paper. We denote the cardinality of a set $A$ by $|A|$. We say that two
eventually nonvanishing sequences $(a_n)$, $(b_n)$ are asymptotically
equivalent and write $a_n\sim b_n$ if $a_n/b_n \to 1$ as
$n\to\infty$. We write $a_n \gg b_n$ if $\lim_{n\to \infty} a_n/b_n = \infty$. 

\medskip

The rest of the paper is organized as follows. In
\autoref{sec:main_result}, we present our main results, which are more
general than the discussion so far. To prepare for the proof, we
establish preliminary results in 
\autoref{sec:general_lemma}. The proofs of the propositions are given
in \autoref{sec:proof_prop}, while \autoref{sec:proof_thm1} and
\autoref{sec:proof_thm2} contain the proofs of our two main
theorems. Finally, \hyperref[appn]{Appendix} contains the proofs of
the preliminary lemmas.

\section{Main results}\label{sec:main_result}

We will impose certain assumptions on the 
failure set \(\Gamma\).

\begin{assumption} \label{ass:failure.set}
 The failure set \(\Gamma\)  is a measurable set bounded
away from the origin. That is, for some $\delta_0>0$, $\Gamma \subset
B_0(\delta_0)^c$, where $B_x(r)$ denoted a radius $r$ closed ball in $\bbr^d$
centered at $x$. We assume that 
\begin{equation}\label{eq:def_Gamma} 
  (\nu\circ \mathcal{A}^{-1})(\Gamma) = \nu(\mathcal{A}^{-1} \Gamma)
  >0, \   (\nu\circ \mathcal{A}^{-1})( \partial \Gamma) = \nu(\partial (\mathcal{A}^{-1} \Gamma)) = 0.
\end{equation}
\end{assumption}

We will see in the sequel  that the first assumption in \eqref{eq:def_Gamma} guarantees that a single
noise vector is involved into ``the catastrophe'' causing the rare
event \( B_0(n, \Gamma) \).
If this assumption fails, multiple noise
vectors may be needed. The second assumption in \eqref{eq:def_Gamma}
is a continuity assumption, that simplifies the statements of our
results.

Our first result can be viewed as precisely demonstrating the action
of the ``catastrophe principle'' in the situation considered in this
paper. The random inputs into the system are the noise random vectors
$(Z_i)_{i\in \ZZ}$, and Theorem \ref{thm:ptprocess} shows that, under
its assumptions, when
the rare event \( B_0(n, \Gamma) \) occurs, exactly one noise 
vector causes it. 

Let $M$ be the space of all Radon measures on
$E=\bbr\times \bigl([-\infty,\infty]^d\setminus \{0\}\bigr)$, equipped with the topology of vague
convergence. We formulate our first result in the language of weak
convergence in $M$ or, even more precisely, weak convergence in $M_p$,
the vaguely closed subset of $M$ consisting of Radon point processes. 
We refer the reader to \cite{resnick:1987}
for more information on weak convergence in the vague topology.

\begin{theorem}\label{thm:ptprocess} 
Suppose that the set $\Gamma$
satisfies Assumption \ref{ass:failure.set}.   Consider a sequence of point processes on $E$ defined by
\[
N_n=\sum_{i = -\infty}^\infty \epsilon_{(i/n, \, \mathcal{A}Z_i/n)}, \ n=1,2,\ldots.
\]
Then the conditional law of $N_n$ given 
 \( B_0(n, \Gamma) \) converges weakly, as \( n \to \infty \),  to the
 law of a single point point process \( \epsilon_{(U, X)} \) where $U$
 is standard uniform, and 
$X$ is an independent of $U$ random vector in $\bbr^d$, with the law 
\[ 
P(X\in \cdot) = 
 \frac{\nu(\mathcal{A}^{-1}(\Gamma \cap \cdot))}{\nu(\mathcal{A}^{-1} \Gamma)}.\]
\end{theorem}

\begin{remark} \label{rk:unif.time}

In addition to establishing that, for large $n$, a single noise vector is likely to drive the rare event \( B_0(n, \Gamma) \), Theorem \ref{thm:ptprocess} also demonstrates that the location of this vector is approximately uniformly distributed over the interval $(o(n), n + o(n))$. This behavior is characteristic of the short memory regime defined by \eqref{eq:def_coeff1}. Intuitively, short memory ensures that no individual input $Z_i$ has oversized overall
influence: the inputs in the above interval enter the partial sum
$S_n$ with, roughly, the same weight   (equal to $\cA$),  so that they
have, 
roughly, the same chance of causing large deviations
 
\end{remark}

Our second main result answers the question of how large deviations
events $(B_j(n,\Gamma))$ cluster given that the event $B_0(n,\Gamma)$
occurs by describing the limiting distribution of the properly
normalized random vector $(J_n^+,J_n^-)$ defined in
\eqref{e:first.stop}. We state the theorem is a somewhat greater
generality than what is required for this purpose. Denote 
\[S_j^n = \sum_{k = j}^ {n+j-1} X_k, \  \ j\in\bbz,
\] 
so that   $B_j(n, \Gamma) = \{S_j^n\in n \Gamma\}$.
 
\begin{theorem} \label{thm:cluster_size}  
Suppose that the sets $\Gamma$ and
$\Psi$ both satisfy Assumption \ref{ass:failure.set}.  For each \(n\), define  
\[
J_n^{\Psi,+} =  \inf \{ j \geq 0 : S_j^n \in n\Psi^c \}, \ \ J_n^{\Psi,-} =  \sup \{ j <0 : S_j^n \in n\Psi^c \}.
\]  
Then, as \(n \to \infty\), the conditional distribution of
\((J_n^{\Psi,+}/n,\, -J_n^{\Psi,-}/n)\) given \(B_0(n, \Gamma)\)
converges weakly to  the law of a random vector equal to 
\[
  \left\{ \begin{array}{ll}
            (U,1-U) & \ \text{with probability} \ \mu(\Psi), \\
            (0,0) & \ \text{with probability} \ 1-\mu(\Psi),
\end{array} \right. 
\]  
where  $U$ is a standard uniform random variable and 
\[
\mu(\Psi) = \frac{\nu(\mathcal{A}^{-1} (\Psi \cap \Gamma))}{\nu(\mathcal{A}^{-1} \Gamma)}.
\]
In particular, the conditional distribution of
\((J_n^{\Gamma, +}/n,\, -J_n^{\Gamma, -}/n)\) given \(B_0(n, \Gamma)\)
converges weakly to  the law of a random vector $(U,1-U)$. 
\end{theorem}  

\begin{remark} \label{rk:spec.psi}
 The last statement of Theorem \ref{thm:cluster_size} easily follows
 from its general statement by taking  $\Gamma = \Psi$,  since
 $\mu(\Gamma)=1$.   

The scenario described by the theorem is, once again, a reflection of
the ``catastrophe principle''. The event \(B_0(n, \Gamma)\) occurs due
to a single extraordinary noise vector. This noise vector, nearly on
its own, 
may place $S_0^n$ into $n\Psi$. If it does, then $S_j^n$ will stay in
$n\Psi$ for a roughly linear number of $j$, depending on the location
of the extraordinary noise vector. If it does not, no $S_j^n$ is
likely to be in $n\Psi$. 
 \end{remark}  

We close this section with stating a key proposition making very
explicit how the ``catastrophe principle'' works, by exhibiting a
number of key equivalent quantities resulting from this
principle. The equivalence of these quantities will be used repeatedly
in our arguments. The proposition and its argument use the following
representation of the sums $S_j^n$, which follows directly from
\eqref{eq:def_moving_avg}:
\begin{equation} \label{e:S0n}
  S_j^n = \sum_{i = -\infty}^\infty \gamma_{i-j, n} Z_i, \ \ \gamma_{in} =
  \sum_{k=-i}^{n-i-1} A_k, \ i=\ldots, -1,0,1,2,\ldots.
\end{equation} 
 
\begin{proposition}\label{prop:equiProb} 
Let \( M, \delta \) be positive constants, and let  \( I_n\subset  \{0
, \dots, n - 1\}, \, n=1,2,\ldots \) be such that $|I_n|\to\infty$. 
Suppose that the set $\Gamma$
satisfies Assumption \ref{ass:failure.set}. Then, as \( n \to \infty \), the following
quantities are asymptotically equivalent:  
\begin{enumerate}
 \item \( (|I_n|/n) P( S_0^n\in n\Gamma) \)
    \item \( |I_n| P(\mathcal{A} Z_0\in n\Gamma) \)
    \item \( P(\text{there is} \  i\in I_n \ \text{with} \  \mathcal{A}Z_i \in n\Gamma) \)
    \item \( P(S_0^n\in n\Gamma  \ \text{and there is} \  i\in I_n \ \text{with} \  \mathcal{A}Z_i \in n\Gamma) \)
    \item \( P( S_0^n\in n\Gamma,   \ \text{there is} \  i\in I_n \
      \text{with} \  \mathcal{A}Z_i \in n\Gamma \ \text{and for all}
      \  i\not=l  \in
      [-n{M}, n{M}],    \|Z_l\| \leq n\delta ). 
\)
\end{enumerate}

In particular,
\begin{equation} \label{e:LRprob}
\lim_{n\to \infty} \frac{P( S_0^n \in n\Gamma )}{n P( \mathcal{A}Z_0\in n\Gamma)} = 1.
\end{equation}
\end{proposition}

Note that the equivalence of 1., 3. and 4. above implies that
\begin{equation} \label{e:almost.same}
\lim_{n\to \infty} \frac{P\bigl( \{S_0^n \in n\Gamma\}\triangle \{
  \text{there is} \  i\in \{0,\ldots, n-1\} \ \text{with} \  \mathcal{A}Z_i \in n\Gamma\}\bigr)}{n P(
  \mathcal{A}Z_0\in n\Gamma)} = 0. 
\end{equation}

We strengthen this conclusion in the following corollary.

\begin{corollary} \label{c:range}
For any $0\leq \theta<1$, $\Gamma$ satisfying Assumption
\ref{ass:failure.set}, 
\begin{equation} \label{e:almost.same.range}
\lim_{n\to \infty} \frac{P\bigl( \{S_j^n \in n\Gamma, \, 0\leq j\leq n\theta\}\triangle \{
  \text{there is} \  i\in \{\lfloor n\theta\rfloor\ldots, n-1\} \ \text{with} \  \mathcal{A}Z_i \in n\Gamma\}\bigr)}{n P(
  \mathcal{A}Z_0\in n\Gamma)} = 0. 
\end{equation}
\end{corollary}

\begin{remark} \label{rk:equiv}
The conclusion \eqref{e:LRprob} is an immediate consequence of the 
equivalence of 1. and 2. above. Furthermore, by the regular variation
(see \eqref{eq:def_regular_varying_Rd}), it can restated as 
\[
\frac{P(S_0^n\in n\Gamma)}{nP(\|\mathcal{A}Z_0\|\geq n)}\to \frac{\nu(\mathcal{A}^{-1}\Gamma)}{\nu(y:\|\cA y\|\geq 1)},
\]
where $\nu$ is the tail measure of the noise. If 
\(\Gamma\) is not a continuity set with respect to \(\nu \circ
\mathcal{A}^{-1}\), this statement should be reformulated as saying 
\[
\frac{\nu(\cA^{-1}\interior{\Gamma})}{\nu(y:\|\cA y\|\geq 1) }
\leq \liminf_{n\to \infty}  \frac{P(S_0^n\in n\Gamma)}{nP(\|\mathcal{A}Z_0\|\geq n)}
\leq \limsup_{n\to \infty}  \frac{P(S_0^n\in n\Gamma)}{nP(\|\mathcal{A}Z_0\|\geq n)}
\leq \frac{\nu(\cA^{-1}\bar{\Gamma})}{\nu(y:\|\cA y\|\geq 1) }.
\]

\end{remark}

\section{Preliminary calculations}\label{sec:Glemma}\label{sec:general_lemma}

In this section we present a number of preliminary results, stated as
lemmas, needed to prove Proposition \ref{prop:equiProb}, its corollary
and the two 
main theorems. Many of these results shed additional light on the 
``catastrophe principle''. 
Some of the results are stated in the one-dimensional case, and then
applied component-wise in the sequel.  

Recall that for a one-dimensional random variable $Z$, its distribution is said to be regularly varying with index $\alpha > 0$ and if there exists a slowly varying function $L: \RR_+ \to \RR_+$ and a constant $p \in [0,1]$ (with $q = 1 - p$) such that

\begin{equation}\label{eq:tail_balance}
  \lim_{z\to\infty}\frac{ P( Z > z)}{z^{-\alpha} L(z)}= p,  \  \
  \lim_{z\to\infty}\frac{ P(Z <- z)}{z^{-\alpha} L(z)}= q. 
\end{equation}

The following variance estimate for the truncated regularly varying
random variables will be used for in the sequel concentration
inequalities. 

\begin{lemma} \label{lemma:Varbound}  Let $Z$  be
  one-dimensional 
  regularly varying with index $\alpha>0$. For $x>0$ let 
$Z^*_x=Z\one(|Z|\leq x)$ . Then, as a function of $x$, $Var(Z^*_x)$ is
bounded by $E(Z^2)$  if $\alpha>2$ and is regularly varying at infinity with
index $2-\alpha$ if $\alpha \leq 2$. 
\end{lemma}
\begin{proof}
The statement for $\alpha>2$ is obvious, while the statement for
$\alpha\leq 2$ follows from the Karamata theorem; see
e.g. \cite{resnick:1987}. 
\end{proof}

The following lemma gives useful bounds on the tails of infinite
sums.  
\begin{lemma}  \label{lemma:WLLN_bounded}
Let $(Z_i)$  be i.i.d.  zero mean one-dimensional random variables that are 
  regularly varying with index $\alpha>1$. For each $n=1,2,\ldots$ let
  $\{\beta_{in}\}_{i =
    -\infty}^\infty $ be a sequence satisfying $\sum_{i =
    -\infty}^\infty |\beta_{in}| \leq nL$ for some $L > 0$ independent
  of $n$. If $\alpha\leq 
  2$, assume also that for some $B>0$ we have $|\beta_{in}|\leq B$ for
  all $i,n$. 
(a) \  For any $\delta, \tau> 0$ and $0<\vep\leq 1$, for all $n\geq
n_0(\vep,\tau/B)$ (that may also depend on the distribution of $Z$), we have 
    \[
    P\left( \left| \sum_{i=-\infty}^\infty \beta_{in} Z_i
        \one(|\beta_{in}Z_i|\leq \tau n)\right| > n\delta\right)  \leq
    \left\{ \begin{array} {ll}
    \left( \frac{n^2\tau\delta}{ E(Z_0^2)\sum_{i=-\infty}^\infty
        \beta_{in}^2}\right)^{-\delta/(8\tau)} & \text{if} \
                                                 \alpha>2,\\
       \left(\frac{n^2\tau\delta  B^{\alpha-2-\vep}} {g(\tau n/B) \sum_{i=-\infty}^\infty
        |\beta_{in}|^{\alpha-\vep}}\right)^{-\delta/(8\tau)} & \text{if} \
                                                 1<\alpha\leq 2,
\end{array} \right.
    \]
    where $g$ is regularly varying with exponent $2-\alpha$ and it
    depends only on the law of $Z$ and $\vep$.

(b) \  For any $\delta, \tau> 0$ and $0<\vep\leq 1$, for all $n\geq
n_0(\vep,\tau/B)$ (that may also depend on the distribution of $Z$), we have 
    \begin{align*}
    P\left( \left| \sum_{i=-\infty}^\infty \beta_{in} Z_i
       \right| > n\delta\right)  \leq& \left( \frac{n^2\tau\delta}{ E(Z_0^2)\sum_{i=-\infty}^\infty
                                       \beta_{in}^2}\right)^{-\delta/(8\tau)} \\
                                       +& 2 B^{-\alpha +\vep} P(|Z_0|>n\tau/B) \sum_{i}
     |\beta_{in}|^{\alpha-\vep}
       \end{align*}
       if $\alpha>2$, and
    \begin{align*}
    P\left( \left| \sum_{i=-\infty}^\infty \beta_{in} Z_i
       \right| > n\delta\right)  \leq&     
  \left(\frac{n^2\tau\delta  B^{\alpha-2-\vep}} {g(\tau n/B) \sum_{i=-\infty}^\infty
      |\beta_{in}|^{\alpha-\vep}}\right)^{-\delta/(8\tau)} \\
+& 2 B^{-\alpha +\vep} P(|Z_0|>n\tau/B) \sum_{i}
     |\beta_{in}|^{\alpha-\vep} 
\end{align*}
if $1<\alpha<2$, where $g$ is as in part (a). 
\end{lemma}

We now present several useful results in the multivariate case, and 
we return to using the notation of the previous section. 
   The following elementary fact is very useful, and we record it for
repeated reference.

\begin{lemma} \label{l:elem.fact}
  For any $\delta>0$ and $M>0$,
\begin{equation} \label{e:two.big.sum}
\lim_{n\to\infty}\frac{\sum_{\substack{i,j=-Mn, \ldots, Mn\\i\not=j}}
  P(\|Z_i\|>n\delta, \|Z_j\|>n\delta)}{P(S_0^n\in n\Gamma)}=0. 
\end{equation}
\end{lemma}
\begin{proof}
It is enough to notice that the numerator in \eqref{e:two.big.sum} is
regularly varying with exponent $2-2\alpha$, while the denominator is,
by the equivalence of 1 and 2 in \autoref{prop:equiProb}, 
regularly varying with exponent $1-\alpha$ (note that this lemma is
not used in the proof of the above equivalence). 
\end{proof}

We now present several results that will be useful for establishing convergence
of the point processes. 
\begin{lemma}  \label{coro:sum_exp_equi}
Let  $f: \mathbb{R}\times\mathbb{R}^d \to \mathbb{R}_+$ be a
nonnegative measurable function, and let $\delta, M$ be positive
constants. For $n\geq 1$ and $i=0,1,\ldots, n-1$ 
consider the events 
\begin{equation} \label{e:event.i}
B^{(i,n)}_{0}(\Gamma) = \bigl\{S_0^n\in
 n\Gamma, \, \cA  Z_i \in n \Gamma, \, \|Z_l\|\leq n\delta \ \text{for
   all} \ l \in [-nM, nM], \, l\not=i\bigr\}.
 \end{equation}
Then, as $n\to\infty$, 
    \begin{align*}
         \frac{E\left[\sum_{i=0}^{n-1}\exp\left(-
      f(i/n, \cA Z_i/n)\right)1_{B^{(i,n)}_{0}(\Gamma)}\right]}{P(S_0^n\in n\Gamma)}
         - 
\frac{ E \left[\sum_{i=0}^{n-1}\exp\left(-
      f(i/n, \cA Z_i/n)\right) 1_{\cA  Z_i\in n\Gamma}\right] }{P(S_0^n\in n\Gamma)} \to 0.
    \end{align*} 
\end{lemma}

\begin{proof}
  We have
  \begin{align*}
&\left|   \frac{E\left[\sum_{i=0}^{n-1}\exp\left(-
      f(i/n, \cA Z_i/n)\right)1_{B^{(i,n)}_{0}(\Gamma)}\right]}{P(S_0^n\in n\Gamma)}
         - 
\frac{ E \left[\sum_{i=0}^{n-1}\exp\left(-
      f(i/n, \cA Z_i/n)\right) 1_{\cA  Z_i\in n\Gamma}\right] }{P(S_0^n\in n\Gamma)} \right|\\
\leq& \frac{E\sum_{i=0}^{n-1} |1_{B^{(i,n)}_{0}(\Gamma)}-1_{\cA  Z_i\in 
      n\Gamma}|}{P(S_0^n\in n\Gamma)} 
\leq \frac{\sum_{i=0}^{n-1} P\bigl(\cA  Z_i\in n\Gamma, S_0^n\notin
 n\Gamma\bigr)}{P(S_0^n\in n\Gamma)} \\
+& \frac{\sum_{i=0}^{n-1} P\bigl(\cA  Z_i\in n\Gamma,  \|Z_l\|> n\delta \ \text{for
   some} \ l \in [-nM, nM], \, l\not=i\bigr)}{P(S_0^n\in n\Gamma)}. 
    \end{align*}
    Note that using inclusion exclusion formula, we have
    \begin{align*}
&\frac{\sum_{i=0}^{n-1} P\bigl(\cA  Z_i\in n\Gamma, S_0^n\notin
 n\Gamma\bigr)}{P(S_0^n\in n\Gamma)} \leq \frac{P(S_0^n\notin n\Gamma
                     \ \text{and there is} \  i=0,\ldots, n-1 \ \text{with} \  \mathcal{A}Z_i \in n\Gamma)}{P(S_0^n\in n\Gamma)}\\
+&\frac{\sum_{\substack{i,j=0,\ldots, n-1\\i\not=j}} P\bigl(\cA  Z_i\in n\Gamma, \cA  Z_j\in n\Gamma, S_0^n\notin
 n\Gamma\bigr)}{P(S_0^n\in n\Gamma)}\to 0
    \end{align*}
 by the equivalence of 3. and 4. in    \autoref{prop:equiProb} and
 Lemma \ref{l:elem.fact}.
 Similarly,
 $$
 \frac{\sum_{i=0}^{n-1} P\bigl(\cA  Z_i\in n\Gamma,  \|Z_l\|> n\delta \ \text{for
   some} \ l \in [-nM, nM], \, l\not=i\bigr)}{P(S_0^n\in n\Gamma)}\to 0
$$
by  Lemma \ref{l:elem.fact}. 
  
\end{proof}

 \begin{lemma}\label{lemma:equiExp2}
Let  $f: \mathbb{R}\times\mathbb{R}^d \to \mathbb{R}_+$ be a
nonnegative continuous function. Then 
\begin{equation} \label{e:ave.cond}
\lim_{n\to\infty} \frac{1}{n} \sum_{i=0}^{n-1}E\bigl[\exp\bigl(-f(i/n,
\cA Z_i/n)\bigr)\big| 
\cA    Z_i \in n\Gamma\bigr] = E[\exp(-f(U, X))], 
\end{equation}
where $U$ is a standard uniform random variable independent of a random
vector $X$ with the law $ \nu\bigl(\cA^{-1} (\cdot \cap
\Gamma)\bigr)/\nu(\cA^{-1}\Gamma)$. 
\end{lemma}
    \begin{proof}
Let $X_{n, i}, \, i=0,\ldots, n-1 $ be independent and identically
distributed random vectors with the law 
$P(X_{n, 0}\in \cdot) = P({\cA Z_0} \in n\cdot|\cA    Z_0\in n
\Gamma)$, and let $U_n$, be an independent of  them random variable
with the discrete uniform distribution on $\{0,\ldots, 
n-1\}$. Then the expression in the left-hand side of
\eqref{e:ave.cond} is simply
$$
E\bigl[\exp\bigl(-f(U_n/n, X_{n, U_n}/n)\bigr)\bigr]. 
$$ 
Since  $U_n /n\Rightarrow U$ and $X_{n, 0}/n\Rightarrow X$, it follows
by the identity of the distributions and independence that $(U_n/n,
X_{n, U_n}/n) \eid (U_n/n,
X_{n, 0}/n) \Rightarrow   (U, X)$, and the claim of the lemma follows. 
\end{proof}

\section{Proof of Proposition \ref{prop:equiProb} and Corollary \ref{c:range}
}\label{sec:proof_prop}
We use
the usual notions of a width-$\epsilon>0$ neighborhood of a set
$\Gamma\subset\RR^d$:  the outer neighborhood defined by
$\Gamma^\epsilon = \{y: dist(\Gamma, y) \leq \epsilon\}$ and the inner
neighborhood defined by  
$\Gamma^{-\epsilon} = ((\Gamma^c)^\epsilon)^c$.
Clearly,
    $\Gamma^{-\epsilon} \subset \Gamma \subset \Gamma^{\epsilon}$, 
  $\Gamma^{\epsilon} \downarrow \bar{\Gamma}$ as $\eps \downarrow 0$
  and  $\Gamma^{-\epsilon} \uparrow \interior{\Gamma}$ as $\eps
  \downarrow 0$. 
We prove first the equivalence of 1. and 2. We represent $S_0^n$ in
the form 
\begin{equation} \label{e:split.S_n}
S_0^n = \sum_{i=0}^{n-1} \cA Z_i + \left( S_n-\sum_{i=0}^{n-1} \cA
  Z_i\right) := C_n+D_n,
\end{equation} 
and we think of $C_n$ as the ``main part'' of the sum and of $D_n$ as
the ``secondary part'' of the sum. 

The behavior of the main part is known. It follows from Theorem 2.1
in \cite{FunctionalLDP} that for any measurable set $\Psi$ that is
bounded away from the origin, 
\begin{equation} \label{e:main.p}
\frac{\nu(\cA^{-1} \Psi^\circ)}{\nu(\cA^{-1}  \Gamma)}\leq  \liminf_{n\to \infty} \frac{P( C_{n }\in n\Psi)}{nP( \mathcal{A}
    Z_0\in n\Gamma)} \leq \limsup_{n\to \infty} \frac{P( C_{n }\in n\Psi)}{nP( \mathcal{A}
    Z_0\in n\Gamma)} \leq \frac{\nu(\cA^{-1} \bar \Psi)}{\nu(\cA^{-1}
    \Gamma)}. 
\end{equation}
 The rest of the argument is designed to show that the secondary
 part is asymptotically negligible. Specifically, we show that for any
 $\epsilon>0$, 
 \begin{equation} \label{e:sec.p}
 \lim_{n\to \infty} \frac{P( \|D_n\| >
   n\epsilon)}{nP( \cA Z_0 \in n\Gamma)} = 0. 
 \end{equation}
The continuity assumption on the set $\Gamma$ shows that the
equivalence of 1. and 2. of the proposition is an immediate conclusion
from \eqref{e:main.p} and \eqref{e:sec.p}.

We can write
$$
D_n = \sum_{i=-\infty}^\infty  \zeta_{in} Z_i
$$
for some $d\times d$ matrices $(\zeta_{in})$. On the event $\{ \|D_n\| >
n\epsilon\}$ there exist $m,k\in \{1,\ldots, d\}$ such that
\begin{equation}
    \label{eq:reduce_1D}
  \left|\sum_{i=-\infty}^\infty  (\zeta_{in})_{mk} (Z_i)_k\right| >
  \frac{n\eps}{d^{3/2}}, 
  \end{equation}
where $(A)_{mk}$ denotes the entry of a matrix
$A$ in position $(m,k)$ and $(z)_k$ denotes the $k$th entry of a vector
$z$. 

We will show that for every $i,m,k$, $|(\zeta_{in})_{m,k}| \leq
A_{abs}$ and for every $m,k$, 
\begin{equation}\label{eq:coef_tot_bdd}
    \sum_i |(\zeta_{in})_{m,k}| \leq \sum_{i} \|\zeta_{in}\| =o(n), \ n\to\infty.
  \end{equation}
 Assuming that this is true, we will  use Lemma \ref{lemma:WLLN_bounded}
with  $B = A_{abs}$. Suppose first that $\alpha>2$. Choose  
$0<\tau<\eps/(8(\alpha-1)d^{3/2})$ and let $0<\theta\leq 1$.  
We have 
\begin{align*}
   & P(\|D_n\|> n\eps) \leq \sum_{m = 1}^d\sum_{k = 1}^d  P\left(
     \left|\sum_{i=-\infty}^\infty  (\zeta_{in})_{mk} (Z_i)_k\right| > 
  \frac{n\eps}{d^{3/2}}\right) \\
    \leq &\sum_{m = 1}^d\sum_{k = 1}^d  \left[\left( \frac{n^2\tau(\eps/d^{3/2})}{ E(Z_0^2)\sum_{i=-\infty}^\infty
     |(\zeta_{in})_{m,k}|^2}\right)^{-\delta/(8\tau)}  
                                       + 2 B^{-\alpha +\theta} P\left( |Z_0|>\frac{n\tau}{B}\right) \sum_{i}
     |(\zeta_{in})_{m,k}|^{\alpha-\theta}\right].
\end{align*}
Note that by \eqref{eq:coef_tot_bdd},
$$
\frac{n^2\tau(\eps/d^{3/2})}{ E(Z_0^2)\sum_{i=-\infty}^\infty
  |(\zeta_{in})_{m,k}|^2}\gg n,
$$
and so by the choice of $\tau$ we have
$$
\left( \frac{n^2\tau(\eps/d^{3/2})}{ E(Z_0^2)\sum_{i=-\infty}^\infty
    |(\zeta_{in})_{m,k}|^2}\right)^{-\delta/(8\tau)} = o\bigl(nP( \cA Z_0 \in n\Gamma)\bigr),
$$
since the expression in the left-hand side 
is dominated by a regularly varying function with exponent smaller
than $-(\alpha-1)$, while the  expression in the right-hand side is
regularly varying   with exponent   $-(\alpha-1)$.

Furthermore, by \eqref{eq:coef_tot_bdd},
$$
\sum_{i}
     |(\zeta_{in})_{m,k}|^{\alpha-\theta} \leq A_{abs}^{1-\theta} \sum_{i}
     |(\zeta_{in})_{m,k}| = o(n),
     $$
     so 
     $$
     P\left( |Z_0|>n\tau/B\right) \sum_{i}
     |(\zeta_{in})_{m,k}|^{\alpha-\theta}= o\left( nP( \cA Z_0 \in
       n\Gamma)\right),
     $$
since by the regular variation the two probabilities differ,
asymptotically, by a finite positive multiplicative constant. 
  This proves \eqref{e:sec.p} in the case $\alpha>2$, and the analysis
  in the case $1<\alpha\leq 2$, is entirely similar. 
 
Therefore, it only remains to check
\eqref{eq:coef_tot_bdd}. By \eqref{e:S0n},
$$
\zeta_{in} = \left\{ \begin{array}{ll}
                       \gamma_{in}-\cA & \text{for} \ i\in [0,\ldots, n-1],\\
                       \gamma_{in} & \text{for} \ i\notin [0,\ldots, n-1].
\end{array}
\right.
$$
Notice that 
\begin{align*}
&\sum_{i=0}^{n-1} \|\gamma_{in}-\cA\|
= \sum_{i=0}^{n-1} \left\| \sum_{j=-\infty}^{-(i+1)} A_j +
  \sum_{j=n-i}^\infty A_j\right\| \\
\leq & \sum_{i=0}^{n-1} \sum_{j=-\infty}^{-(i+1)} \|A_j\|
+ \sum_{i=1}^n \sum_{j=i}^\infty \|A_j\|=o(n)
\end{align*}
by the absolute summability assumption
\eqref{eq:def_coeff1}. Similarly, 
\begin{align*}
\sum_{i=-\infty}^{-1} \|\gamma_{in}\|\leq \sum_{i=-\infty}^{-1}
  \sum_{j=-i}^{n-i-1} \| A_j\| = \sum_{j=0}^{n-1}\sum_{i=j+1}^\infty  \|A_i\|=o(n)
\end{align*}
and 
\begin{align*}
\sum_{i=n}^{\infty} \|\gamma_{in}\|\leq \sum_{i=n}^{\infty}
  \sum_{j=-i}^{n-i-1} \| A_j\| = \sum_{j=1}^{n}\sum_{i=-\infty}^{-j}
  \|A_i\|=o(n). 
\end{align*}
Hence \eqref{eq:coef_tot_bdd} follows. 

Next we prove the equivalence of 2. and 3. 
 Denote $C_i = \{ \cA Z_i \in n\Gamma\}$, so that 
\begin{align*}
    P(\cup_{i\in I_n} C_i) \leq   |I_n|P(C_0). 
\end{align*}
For the bound in the other direction, by the inclusion exclusion
formula and independence, 
\begin{align*}
    & P(\cup_{i\in I_n} C_i) \geq  |I_n|P(C_0) -   (|I_n|(|I_n| - 1)/2) P(C_0)^2\\
=& |I_n| P(C_0) \left[1-  (|I_n|-1) P\left(C_0\right)/2\right]. 
\end{align*}
Since $P(C_0)$ is regularly varying in $n$ with exponent $-\alpha<-1$
and  $|I_n|\leq n$, the expression in the square brackets converges to
1, and the equivalence of 2. and 3. follows.

We now check the equivalence of 3. and 4. Write 
\begin{align} \label{e:1.enough}
&1-\frac{P(S_0^n\in n\Gamma  \ \text{and there is} \  i\in I_n \
  \text{with} \  \mathcal{A}Z_i \in n\Gamma)}{P(\text{there is} \
  i\in I_n \ \text{with} \  \mathcal{A}Z_i \in n\Gamma)} \\
\notag \leq& \frac{\sum_{i\in I_n}P(S_0^n\notin n\Gamma, \, \mathcal{A}Z_i \in n\Gamma)}{P(\text{there is} \
      i\in I_n \ \text{with} \  \mathcal{A}Z_i \in n\Gamma)}
      \sim \frac{\sum_{i\in I_n}P(S_0^n\notin n\Gamma, \,
      \mathcal{A}Z_i \in n\Gamma)}{|I_n| P (\mathcal{A}Z_0 \in
      n\Gamma)}, 
\end{align}
where we have used the already established equivalence of 1. and 3.  

Let $d_n\to\infty$, $d_n/|I_n|\to 0$ be a sequence of positive
numbers. For any $\eps>0$, 
\begin{align*}
  &\sum_{i\in I_n}P(S_0^n\notin n\Gamma, \, \mathcal{A}Z_i \in n\Gamma^{-\eps})
  \leq \sum_{i\in I_n\cap[d_n, n-d_n]}P(S_0^n\notin n\Gamma, \, \mathcal{A}Z_i \in n\Gamma^{-\eps})\\
   +& 2d_nP(\mathcal{A}Z_0 \in n\Gamma^{-\eps}).
\end{align*}
Note that for every $i$, 
\begin{align*}
&P(S_0^n\notin n\Gamma, \, \mathcal{A}Z_i \in n\Gamma^{-\eps}) \\
\leq &P(S_0^n\notin n\Gamma, \, \gamma_{in}Z_i \in n\Gamma^{-\eps/2}) 
+ P(\mathcal{A}Z_i \in n\Gamma^{-\eps}, \, \gamma_{in}Z_i \notin
  n\Gamma^{-\eps/2})\\
\leq &P\left( \left\| \sum_{l\not=i}
       \gamma_{ln}Z_l\right\|>n\eps/2\right)P\bigl( \gamma_{in}Z_i \in n\Gamma^{-\eps/2}\bigr)
+ P\bigl( \| (\cA -\gamma_{in})Z_i\|>n\eps/2\bigr)\\
\leq &P\left( \left\| \sum_{l\not=i}
       \gamma_{ln}Z_l\right\|>n\eps/2\right)P\bigl( \cA Z_i \in n\Gamma\bigr)
+ 2P\bigl( \| (\cA -\gamma_{in})Z_i\|>n\eps/2\bigr).
\end{align*}
Since
\begin{align*}
\frac{2d_nP(\mathcal{A}Z_0 \in n\Gamma^{-\eps})}{|I_n| P (\mathcal{A}Z_0 \in
      n\Gamma)} \leq 2d_n/|I_n|\to 0,
\end{align*}
\begin{align*}
&\frac{\sum_{i\in I_n\cap[d_n, n-d_n]} P\left( \left\| \sum_{l\not=i}
       \gamma_{ln}Z_l\right\|>n\eps/2\right)P\bigl( \cA Z_i \in
  n\Gamma\bigr)}{|I_n| P (\mathcal{A}Z_0 \in       n\Gamma)} \\
\leq& \frac{1}{|I_n|}\sum_{i\in I_n\cap[d_n, n-d_n]} P\left( \left\| \sum_{l\not=i}
       \gamma_{ln}Z_l\right\|>n\eps/2\right) \to 0
\end{align*}
because
\begin{align*}
&P\left( \left\| \sum_{l\not=i}
    \gamma_{ln}Z_l\right\|>n\eps/2\right) =
P\bigl(\big\| S_0^n - \gamma_{in}Z_i\big\|>n\eps/2\bigr)\\
\leq& P\bigl(\big\| S_0^n \big\|>n\eps/4\bigr)+ P(A_{abs}\|
      Z_0\|>n\eps/4)\to 0, 
\end{align*}
and 
\begin{align*}
  &\frac{\sum_{i\in I_n\cap[d_n, n-d_n]}P\bigl( \| (\cA -\gamma_{in})Z_i\|>n\eps/2\bigr)}
{|I_n| P (\mathcal{A}Z_0 \in       n\Gamma)} \\
\leq& \frac{P\left( \left( \sum_{|i|\geq d_n}
       \|A_i\|\right)\|Z_0\| >n\eps/2\right)}{P (\mathcal{A}Z_0
       \in       n\Gamma)} \to 0
\end{align*}
by the absolute summability assumption \eqref{eq:def_coeff1}, we
conclude that
\begin{equation} \label{e:with.eps}
\frac{\sum_{i\in I_n}P(S_0^n\notin n\Gamma, \,
      \mathcal{A}Z_i \in n\Gamma^{-\eps})}{|I_n| P (\mathcal{A}Z_0 \in
      n\Gamma)} \to 0
  \end{equation}
  for any $\eps>0$. Since
  \begin{align*}
&\left| \sum_{i\in I_n}P(S_0^n\notin n\Gamma, \,
      \mathcal{A}Z_i \in n\Gamma)- \sum_{i\in I_n}P(S_0^n\notin n\Gamma, \,
      \mathcal{A}Z_i \in n\Gamma^{-\eps})\right|\\
\leq& |I_n| P\bigl(\cA Z_0\in n(\Gamma\setminus \Gamma^{-\eps})\bigr),
  \end{align*}
  it follows from \eqref{e:with.eps} that
  \begin{align*}
&\limsup_{n\to\infty} \frac{\sum_{i\in I_n}P(S_0^n\notin n\Gamma, \,
      \mathcal{A}Z_i \in n\Gamma)}{|I_n| P (\mathcal{A}Z_0 \in
      n\Gamma)} \leq \limsup_{n\to\infty} \frac{P \bigl(\mathcal{A}Z_0 \in
      n(\Gamma\setminus \Gamma^{-\eps})\bigr)}{P (\mathcal{A}Z_0 \in
      n\Gamma)}  \\
\leq& \frac{\nu\bigl( \cA^{-1}( {\bar \Gamma}\setminus
    \Gamma^{-\eps})\bigr)}{\nu(\Gamma)} \to 0
  \end{align*}
 as $\eps\to 0$ by the assumption of the $\nu$-continuity of the set
 $\cA^{-1}\Gamma$.  Together with \eqref{e:1.enough} this establishes
 the equivalence of 3. and 4. 

Finally, we check the equivalence of 4. and 5. Let $L = [-n{M} , n{M}
]\cap \mathbb{N}$. Note that 
\begin{align*}
&1-\frac{P( S_0^n\in n\Gamma,   \ \text{there is} \  i\in I_n \
      \text{with} \  \mathcal{A}Z_i \in n\Gamma \ \text{and for all} \ 
        i\not=l\in L,    \|Z_l\| \leq n\delta )}{P(S_0^n\in n\Gamma  \
  \text{and there is} \  i\in I_n \ \text{with} \  \mathcal{A}Z_i \in
  n\Gamma)} \\
\leq&  \frac{\sum_{i\in I_n}\sum_{l\in L, \, l\not= i} P( \cA  Z_i\in n
      \Gamma)P(\|Z_l\| > n\delta)}{P(S_0^n\in n\Gamma  \
  \text{and there is} \  i\in I_n \ \text{with} \  \mathcal{A}Z_i \in
  n\Gamma)}, 
\end{align*}
so by the already established equivalence of 
2. and 4., it suffices to show that 

$$\lim_{n\to \infty} 2nM P(\|Z_l\| > n\delta) = 0.$$
This is, however, immediate since  the probability in question is
regularly varying with exponent $-\alpha<-1$.

This completes the proof of Proposition \ref{prop:equiProb}. \qed

\medskip

We now prove Corollary \ref{c:range}. To start with, observe that by
\eqref{e:almost.same}, 
\begin{align*}
&P\bigl( \{S_j^n \in n\Gamma, \, 0\leq j\leq n\theta\}\setminus \{
  \text{there is} \  i\in \{\lfloor n\theta\rfloor,\ldots, n-1\} \
  \text{with} \  \mathcal{A}Z_i \in n\Gamma\}\bigr) \\
\leq&P\bigl( \{S_0^n \in n\Gamma,  S_{\lfloor n\theta\rfloor}^n \in
      n\Gamma\}\setminus \{ 
  \text{there is} \  i\in \{\lfloor n\theta\rfloor,\ldots, n-1\} \
  \text{with} \  \mathcal{A}Z_i \in n\Gamma\}\bigr) \\
=&  P\bigl( \{\text{there is} \  i\in
  \{0,\ldots, n-1\} \ \text{with} \  \mathcal{A}Z_i \in n\Gamma 
\ \   \text{and there is} \ i\in
  \{\lfloor n\theta\rfloor,\ldots, \lfloor n\theta\rfloor +n-1\} \\ 
 &  \hskip 0.2in \text{with} \  \mathcal{A}Z_i \in n\Gamma \} \setminus \{
  \text{there is} \  i\in \{\lfloor n\theta\rfloor,\ldots, n-1\} \
  \text{with} \  \mathcal{A}Z_i \in n\Gamma\}\bigr) \\
+&
o\bigl( n P(
  \mathcal{A}Z_0\in n\Gamma)\bigr) = o\bigl( n P(
  \mathcal{A}Z_0\in n\Gamma)\bigr), 
\end{align*}
where on the last step we used the equivalence of 2. and 3. in
Proposition \ref{prop:equiProb}.

For the second bound in the corollary we use an approach analogous to
the proof of equivalence of 3. and 4. in Proposition
\ref{prop:equiProb}. Write
\begin{align*}
&P\bigl(   \{
  \text{there is} \  i\in \{\lfloor n\theta\rfloor,\ldots, n-1\} \
  \text{with} \  \mathcal{A}Z_i \in n\Gamma\} \setminus \{S_j^n \in n\Gamma, \, 0\leq j\leq n\theta\}\bigr) \\
\leq& \sum_{i =\lfloor n\theta\rfloor}^{n-1} P\bigl( \cA Z_i \in  n\Gamma, \,  S_j^n
  \notin n\Gamma \ \ \text{for some} \ 0\leq j\leq n\theta\bigr).
\end{align*}
Let $\vep>0$ and let $d_n\to\infty, \, d_n/n\to 0$. We have
\begin{align*}
&\sum_{i =\lfloor n\theta\rfloor}^{n-1} P\bigl( \cA Z_i \in  n\Gamma^{-\vep}, \,  S_j^n
  \notin n\Gamma \ \ \text{for some} \ 0\leq j\leq n\theta\bigr)\\
\leq& \sum_{i =\lfloor n\theta\rfloor + d_n}^{n-d_n} P\bigl( \cA Z_i \in  n\Gamma^{-\vep}, \,  S_j^n
  \notin n\Gamma \ \ \text{for some} \ 0\leq j\leq n\theta\bigr)\\
+&2d_nP\bigl(\cA Z_0 \in  n\Gamma^{-\vep}\bigr). 
\end{align*}
Next, for every $i$,
\begin{align*}
&P\bigl( \cA Z_i \in  n\Gamma^{-\vep}, \,  S_j^n
  \notin n\Gamma 
 \ \ \text{for some} \ 0\leq j\leq n\theta\bigr)\\
\leq& P\bigl(  S_j^n
  \notin n\Gamma, \, \gamma_{i-j,n}Z_i\in n\Gamma^{-\vep/2} 
 \ \ \text{for some} \ 0\leq j\leq n\theta\bigr)\\
+& P\bigl( \cA Z_i \in  n\Gamma^{-\vep}, \, \gamma_{i-j,n}Z_i\notin
      n\Gamma^{-\vep/2}   \ \ \text{for some} \ 0\leq j\leq
      n\theta\bigr) \\
\leq& P\left( \left\| \sum_{l\not=i}
      \gamma_{l-j,n}Z_l\right\|>n\vep/2, \, \gamma_{i-j,n}Z_i\in n\Gamma^{-\vep/2} 
 \ \ \text{for some} \ 0\leq j\leq n\theta\right)\\
+& P\bigl( \| (\cA -\gamma_{i-j,n})Z_i\|>n\vep/2 \ \ \text{for some} \
   0\leq j\leq n\theta\bigr) \\
\leq& P\left( \left\| \sum_{l\not=i}
      \gamma_{l-j,n}Z_l\right\|>n\vep/2 \ \ \text{for some} \ 0\leq
      j\leq n\theta\right)P(\cA Z_i\in n\Gamma) \\
+& 2P\bigl( \| (\cA -\gamma_{i-j,n})Z_i\|>n\vep/2 \ \ \text{for some} \
   0\leq j\leq n\theta\bigr).  
\end{align*}
We conclude as in the proof of equivalence of 3. and 4. in Proposition
\ref{prop:equiProb}, by noticing that
\begin{align*}
&P\left( \left\| \sum_{l\not=i}
      \gamma_{l-j,n}Z_l\right\|>n\vep/2 \ \ \text{for some} \ 0\leq
      j\leq n\theta\right) \\
\leq& P\bigl( \| S_j^n \|>n\vep/4 \ \ \text{for some} \ 0\leq
      j\leq n\theta\bigr) + P(A_{abs} \|Z_0\|>n\vep/4)\to 0
  \end{align*}
  by the ergodic theorem. \qed

%

\section{Proof of \autoref{thm:ptprocess}}\label{sec:proof_thm1}

In order to show weak convergence of a sequence of Radon measures, it
is enough to prove convergence of certain corresponding Laplace
functionals; see Proposition 3.19 in \cite{resnick:1987}.  That is, 
 the weak convergence claimed in the theorem will follow once we
prove that
\begin{equation} \label{e:Laplace.conv}
E\exp\left[ \left\{ -\sum_{i=-\infty}^\infty f\bigl( i/n,\cA
    Z_i/n\bigr)\right\} \Bigg| B_0(n,\Gamma)\right]
\to E\exp\big\{ -f(U,X)\bigr\}
\end{equation}
for every nonnegative continuous function 
$f: \mathbb{R} \times \mathbb{R}^d \to \mathbb{R}_+$   whose support
is contained in $[-{M}, {M}] \times
B_{0}(\theta)^c$ for some $M,\theta \in (0,\infty)$. 
  
Fix $f$ as above, and abbreviate the left-hand side of
\eqref{e:Laplace.conv} as $E\exp\{ -N_n(f)|B_0(n,\Gamma)\}$. Let
$0<\delta<\theta$ and denote 
\begin{align*}
 B_0^{(1)}(n,\Gamma)= B_0(n,\Gamma) \cap \bigl\{& \text{there is} \
  i\in \{0,\ldots, n-1\} \       \text{with} \  \mathcal{A}Z_i \in n\Gamma \\
& \text{and for all}
      \  i\not=l  \in
      [-n{M}, n{M}],    \|Z_l\| \leq n\delta )\}. 
\end{align*}
It follows immediately from $f\geq 0$ and the equivalence of 1. and 5. in
Proposition \ref{prop:equiProb} that 
\begin{equation}\label{eq:sum_single_large_jump}
 \frac{E\exp\{ -N_n(f)\one(B_0(n,\Gamma))\}}{P(B_0(n, \Gamma))}-\frac{E\exp\{ -N_n(f)\one(B_0^{(1))}(n,\Gamma)\}}{P(B_0(n, \Gamma))}
 \to 0. 
\end{equation}
Next, for $i=0,\ldots, n-1$ let $B_0^{(i,n)}(\Gamma)$ be the event
defined in \eqref{e:event.i}. By the choice of $\delta$ these events
are disjoint and  $N(f)1_{B_0^{(i,n)}(\Gamma)} = f(i/n, \cA Z_i/n)
 1_{B_0^{(i,n)}(\Gamma)}$. Furthermore, $\cup_i
 B_0^{(i,n)}(\Gamma)= B_0^{(1)}(n,\Gamma)$. Therefore, 
 \begin{align*}
 &\frac{E\exp\{ -N_n(f)\one_{B_0^{(1)}(n,\Gamma)}\}}{P(B_0(n,\Gamma))} = \sum_{i=0}^{n-1}
 \frac{E\exp\{ -f(i/n, \cA Z_i/n)\one_{B_0^{(i,n)}(\Gamma)}\} }{P(B_0(n,\Gamma))}\\
=& (1+o(1)) \sum_{i=0}^{n-1}
 \frac{E\exp\{ -f(i/n, \cA Z_i/n)\one_{ \cA Z_i\in n\Gamma}\}
   }{P(B_0(n,\Gamma))} \\
=& (1+o(1)) \sum_{i=0}^{n-1}
 \frac{E\exp\{ -f(i/n, \cA Z_i/n)\one_{ \cA Z_i\in n\Gamma}\}
   }{nP(\cA   Z_0\in n \Gamma) } \to E[\exp(-f(U, X))], 
 \end{align*}
where on the last three steps we
used   \autoref{coro:sum_exp_equi}, 
the equivalence of 1. and 2. in
Proposition \ref{prop:equiProb} and 
Lemma \ref{lemma:equiExp2}. By
\eqref{eq:sum_single_large_jump} we
conclude that
$$
E[\exp\{ -N_n(f)\}|B_0(n,\Gamma)] \to
E[\exp(-f(U, X))],
$$
hence proving
\eqref{e:Laplace.conv}. This
completes the proof of the theorem. \qed

 \section{Proof of \autoref{thm:cluster_size}}\label{sec:proof_thm2}
Analogously to the notation
$B_j(n,\Gamma)$ in \eqref{e:rare.event}, we
will use the notation   $B_j(n,\Psi)= \{S_j^n\in n \Psi\}$. Let
$\theta_1, \theta_2 \geq 0$ be such that $\theta_1 + \theta_2 <
1$. The probability we analyze in the course of the proof is the
(conditional) probability of the intersection of all events
$B_j(n,\Psi)$ for $j$ in the range from $\lfloor -n\theta_2\rfloor+1$
to $\lceil n\theta_1 \rceil-1$. Recall that the sum
$\sum_{i=0}^{n-1} \cA Z_i$ in \eqref{e:split.S_n} collected the noise
variables that may be responsible for the event $B_0(n,\Gamma)$. In
the present case the noise variables $Z_i$ that may be responsible for the above
intersection turn out to be those with $i$ in the range $\{\lfloor
n\theta_1 \rfloor, \cdots, \lfloor n(1- \theta_2) \rfloor \}$. We
collect them in an analogous sum   $T = \cA \sum_{i= \lfloor n\theta_1\rfloor}^{\lfloor
   n(1-\theta_2)\rfloor} Z_i$.  We have 
 \begin{align} \label{e:upper.quad}
    & P\bigl( J_n^{\Psi,+}/n \geq \theta_1,   J_n^{\Psi,-}/n \leq -
       \theta_2\big| B_0(n,\Gamma)\bigr)=  P\left(\bigcap_{j= -\lfloor
       n\theta_2\rfloor +1}^{\lceil n\theta_1\rceil -1} B_j(n,\Psi)\bigg| B_0(n,\Gamma)\right)\\
  \notag    = & P\left(\{ T\in n\Psi\}\cap  \bigcap_{j= -\lfloor
       n\theta_2\rfloor +1}^{\lceil n\theta_1\rceil -1} B_j(n,\Psi)\bigg| B_0(n,\Gamma)\right) \\
+ & P\left(\{ T\notin n\Psi\}\cap \bigcap_{j= -\lfloor
       n\theta_2\rfloor +1}^{\lceil n\theta_1\rceil -1} B_j(n,\Psi)\bigg|
         B_0(n,\Gamma)\right). \notag 
  \end{align}
 
 We start with the following claim.

\begin{equation}\label{claim:small_common_sum}
 P\left(\{ T\notin n\Psi\}\cap \bigcap_{j= -\lfloor
       n\theta_2\rfloor +1}^{\lceil n\theta_1\rceil -1} B_j(n,\Psi)\bigg|
         B_0(n,\Gamma)\right) \to 0, \text{ as } n\to \infty
\end{equation}
For continuity of the presentation we defer the proof of
\eqref{claim:small_common_sum} to the final section.

We now analyze the first probability in the right-hand side of
\eqref{e:upper.quad}. 
We write it as
\begin{align} \label{e:P1.upper}
&P\left(\{ T\in n\Psi\}\cap \bigcap_{j= -\lfloor
       n\theta_2\rfloor +1}^{\lceil n\theta_1\rceil -1}   B_j(n,\Psi)\bigg|
  B_0(n,\Gamma)\right) \\
  \notag =& P\bigl(T\in n\Psi\big| B_0(n,\Gamma)\bigr)
  P\left(\bigcap_{j= -\lfloor
       n\theta_2\rfloor +1}^{\lceil n\theta_1\rceil -1}  B_j(n,\Psi)\bigg|
  \{ T\in n\Psi\}\cap B_0(n,\Gamma)\right).
  \end{align}
    
We now apply the equivalence between 1.,  3. and 4. in Proposition
\ref{prop:equiProb} to the moving average process for which $A_0=\cA$
and $A_i=0$ for $i\not=0$ to conclude that
\begin{align*} 
&P\bigl(T\in n\Psi\big| B_0(n,\Gamma)\bigr) \sim P\bigl(\cA Z_i
  \in n\Psi \ \ \text{for some} \ \ \lfloor n\theta_1\rfloor \leq i
  \leq \lfloor n(1-\theta_2)\rfloor\big| B_0(n,\Gamma)\bigr) \\
 \notag =&P\left( N_n\left(  \bigl[ \lfloor n\theta_1\rfloor/n, \lfloor
    n(1-\theta_2)\rfloor/n\bigr]\times \Psi\right)>0\big|B_0(n,\Gamma)
    \right) \\
\notag \to& P\bigl( \eps_{(U,X)}\bigl([\theta_1,1-\theta_2]\times \Psi\bigr)>0\bigr)    
= (1-\theta_1-\theta_2) \frac{\nu(\mathcal{A}^{-1} (\Psi \cap
     \Gamma))}{\nu(\mathcal{A}^{-1} \Gamma)} 
 \end{align*}
  by Theorem \ref{thm:ptprocess} and continuous mapping theorem. If
  $\nu(\cA^{-1}(\Psi \cap \Gamma)) = 0$, \eqref{e:P1.upper} converges
  to 0, hence so does \eqref{e:upper.quad}. The theorem holds in this
  case. In the following, we assume, therefore, that $\nu(\cA^{-1}(\Psi \cap \Gamma)) > 0$. Under this assumption, if we show that 
 \begin{equation*} 
 P\left(\bigcap_{j= -\lfloor
       n\theta_2\rfloor +1}^{\lceil n\theta_1\rceil -1} B_j(n,\Psi)\bigg|
  \{ T\in n\Psi\}\cap B_0(n,\Gamma)\right) \to 1, 
\end{equation*}
then this, along with \eqref{claim:small_common_sum} and
\eqref{e:upper.quad} will give us
$$
P\bigl( J_n^{\Psi,+}/n \geq \theta_1,   J_n^{\Psi,-}/n \leq-
       \theta_2\big| B_0(n,\Gamma)\bigr) \to (1-\theta_1-\theta_2)
       \frac{\nu(\mathcal{A}^{-1} (\Psi \cap 
     \Gamma))}{\nu(\mathcal{A}^{-1} \Gamma)} 
   $$
 for all $\theta_1,\theta_2\geq 0$  such that $\theta_1+\theta_2<1$,
 which implies the statement of the theorem. We will prove that under
 the assumption $\nu(\cA^{-1}(\Psi \cap \Gamma)) > 0$,  
 \begin{equation*} 
 P\left(\bigcup_{j= -\lfloor
       n\theta_2\rfloor +1}^{\lceil n\theta_1\rceil -1} B_j(n,\Psi)^c\bigg|
  \{ T\in n\Psi\}\cap B_0(n,\Gamma)\right) \to 0, 
\end{equation*}
for which it is enough to prove that
\begin{equation*} 
 \frac{P\left(\{ T\in n\Psi\}\cap \bigcup_{j= -\lfloor
       n\theta_2\rfloor +1}^{\lceil n\theta_1\rceil -1} B_j(n,\Psi)^c\right)}
{P\bigl(  \{ T\in n\Psi\}\cap B_0(n,\Gamma)\bigr)} \to 0.  
\end{equation*}
Since by 
\begin{align*}
   &P\bigl(  \{ T\in n\Psi\}\cap
  B_0(n,\Gamma)\bigr) \\
=& P\bigl(  \cA Z_i\in n\Psi \ \ \text{for
  some} \ i= \lfloor n\theta_1\rfloor, \ldots, \lfloor
   n(1-\theta_2)\rfloor \\
 & \hskip 0.2in  \text{and} \ \ \cA Z_i\in n\Gamma\ \ \text{for
  some} \ i=0,\ldots, n-1\bigr) + o\bigl( nP(\cA Z_0\in n\Gamma)\bigr)\\
  \geq & P(\cA Z_i \in n \bigl(\Psi\cap \Gamma\bigr) \ \ \text{for some} \ i = \lfloor n\theta_1\rfloor, \ldots, \lfloor
   n(1-\theta_2)\rfloor)\\
    &  \hskip 0.2in + o\bigl( nP(\cA Z_0\in n\Gamma)\bigr)
\end{align*}
we conclude that 
$$
\liminf_{n\to\infty} \frac{P\bigl(  \{ T\in n\Psi\}\cap
  B_0(n,\Gamma)\bigr)}{n(1-\theta_1-\theta_2) P(\cA Z_0\in n\bigl(\Psi\cap \Gamma\bigr))}\geq 1.
$$

We will prove that
\begin{equation} \label{e:cond.p.big}
 \frac{P\left(\{ T\in n\Psi\}\cap \bigcup_{j= -\lfloor
       n\theta_2\rfloor +1}^{\lceil n\theta_1\rceil -1} B_j(n,\Psi)^c\right)}
{n P(\cA Z_0\in n\bigl(\Psi\cap \Gamma\bigr))} \to 0.  
\end{equation}
 
Concentrating once again on the  relevant individual noise variables
$Z_i$, we write 
\begin{align} \label{e:split.num3} 
&P\left(\{ T\in n\Psi\}\cap  \bigcup_{j= -\lfloor
       n\theta_2\rfloor +1}^{\lceil n\theta_1\rceil -1}
                 B_j(n,\Psi)^c\right) \leq P\left(\{ T\in n\Psi\}\cap \bigcap_{i =\lceil
    n\theta_1\rceil}^{\lfloor   n(1-\theta_2)\rceil} \bigl\{\cA Z_i
                 \notin n\Psi\bigr\}\right) \\
 +& P\left(\bigcup_{i =\lceil
    n\theta_1\rceil}^{\lfloor   n(1-\theta_2)\rceil} \bigl\{\cA Z_i
                 \in n\Psi\bigr\} \cap \bigcup_{j= -\lfloor
       n\theta_2\rfloor +1}^{\lceil n\theta_1\rceil -1}
                 B_j(n,\Psi)^c\right) =: I_1(n) +I_2(n). 
   \notag
\end{align}
The proof of the theorem will be completed once we prove
\eqref{claim:small_common_sum}, as well  as the claim that  
\begin{equation} \label{e:terms.13}
\lim_{n\to\infty} \frac{I_i(n)}
{n P(\cA Z_0\in n\bigl(\Psi\cap \Gamma\bigr))}=0, \ i=1,2.
\end{equation} 

We begin by proving \eqref{claim:small_common_sum}. We have 
\begin{align*}
     &   P\left(\{ T\notin n\Psi\}\cap \bigcap_{j= -\lfloor
       n\theta_2\rfloor +1}^{\lceil n\theta_1\rceil -1} B_j(n,\Psi)\bigg|
     B_0(n,\Gamma)\right) \\
\leq& \frac{P\bigl(S^n_{-\lfloor
       n\theta_2\rfloor +1}\in n\Psi, \, S^n_{\lceil n\theta_1\rceil-1}
       \in n \Psi, \, T \notin n \Psi\bigr)}{P(B_0(n,\Gamma))}, 
 \end{align*}
so by the equivalence of 1. and 2. in  in Proposition
\ref{prop:equiProb}, it is enough to show that
\begin{align*}
\lim_{n\to\infty} \frac{P\bigl(S^n_{-\lfloor
       n\theta_2\rfloor +1}\in n\Psi, \, S^n_{\lceil n\theta_1\rceil-1}
       \in n \Psi, \, T \notin n \Psi\bigr)}{nP(\cA Z_0\in
  n\Gamma)}=0, 
\end{align*}
which will follow from the claims 
\begin{align} \label{e:23.1}
&\lim_{n\to\infty} \frac{P\bigl(S^n_{-\lfloor
       n\theta_2\rfloor +1}\in n\Psi, \, S^n_{\lceil n\theta_1\rceil-1}
       \in n \Psi, \,  \cA Z_i\notin n\Psi, \,  
  i=  \lceil n\theta_1\rceil-1, \ldots, -\lfloor
       n\theta_2\rfloor +1 +n\bigr)}{nP(\cA Z_0\in
  n\Gamma)}\\
\notag =&0 
\end{align}
and
\begin{align} \label{e:23.2}
\lim_{n\to\infty} \frac{P\bigl( T \notin n \Psi, \,  \cA
  Z_i\in n\Psi \  \ \text{for some} 
  \ \ i= \lceil n\theta_1\rceil-1, \ldots, -\lfloor
       n\theta_2\rfloor +1 +n\bigr) }{nP(\cA Z_0\in
  n\Gamma)}=0. 
\end{align}
In fact, the proof of \eqref{e:23.1} is very similar to the proof of
\eqref{e:terms.13} for $i=1$, and the proof of \eqref{e:23.2} is very
similar to the proof of \eqref{e:terms.13} for $i=2$. 
 
Using \eqref{e:almost.same},
\begin{align*}
&P\bigl(S^n_{-\lfloor
       n\theta_2\rfloor +1}\in n\Psi, \, S^n_{\lceil n\theta_1\rceil-1}
       \in n \Psi, \,  \cA Z_i\notin n\Psi, \,  
  i= \lfloor  n\theta_1\rfloor, \ldots, \lfloor
   n(1-\theta_2)\rfloor\bigr) \\
=&P\bigl( A_{n,-}\cap A_{n,+}\cap \{\cA Z_i\notin n\Psi, \,  
  i= \lfloor  n\theta_1\rfloor, \ldots, \lfloor
   n(1-\theta_2)\rfloor\}\bigr) + o\bigl( nP(\cA Z_0\in
  n\Gamma)\bigr), 
\end{align*}
where 
$$
 A_{n,-}=\bigl\{\text{there is} \  i\in \{-\lfloor
       n\theta_2\rfloor +1,\ldots, -\lfloor
       n\theta_2\rfloor +1 +n\} \ \text{with} \  \mathcal{A}Z_i \in n\Psi\bigr\},
$$
$$
 A_{n,+}=\bigl\{\text{there is} \  i\in \{\lceil n\theta_1\rceil-1
 ,\ldots, \lceil n\theta_1\rceil+n-2 \} \ \text{with} \
 \mathcal{A}Z_i \in n\Psi\bigr\}. 
$$
Letting $I_{n, -}= \{-\lfloor
       n\theta_2\rfloor +1,\ldots, \lfloor n\theta_1\rfloor - 1\}$, $I_{n, +} = \{\lfloor
   n(1-\theta_2)\rfloor + 1, \lceil n\theta_1\rceil+n-2\}$, and
   noticing that these sets are disjoint, 
 we can rewrite the right-hand side above as 
\begin{align*}
    &P(\text{there are} \ i_{-} \in I_{n, -}, \,  i_{+} \in I_{n, +} \
    \text{ with } \ \cA Z_{i_{-}}, \cA Z_{i_{+}} \in n \Psi) + o(nP(\cA Z_0 \in n \Gamma)) \\
 =& o(nP(\cA Z_0 \in n \Gamma)),
\end{align*}
where on the last step we used  the equivalence of 1. and 2. in
Proposition \ref{prop:equiProb} and independence. Therefore,
\eqref{e:23.1} follows. 

Similarly, by \eqref{e:almost.same},
$$P\bigl( T \notin n \Psi, \,  \cA
  Z_i\in n\Psi \  \ \text{for some} 
  \ \ i= \lfloor n\theta_1\rfloor, \ldots, \lfloor
       n(1-\theta_2)\rfloor  \bigr)  = o\bigl( nP(\cA Z_0\in
  n\Gamma)\bigr). 
$$
Since the events $\{\cA
  Z_i\in n\Psi \  \ \text{for some} 
   \ \ i= \lfloor n\theta_1\rfloor, \ldots, \lfloor
       n(1-\theta_2)\rfloor \}$ and  $\{\cA
  Z_i\in n\Psi \  \ \text{for some} 
  \ \ i= \lceil n\theta_1\rceil-1, \ldots, -\lfloor
       n\theta_2\rfloor +1 +n\}$ differ by at most 2 events of the
       type $\{\cA Z_i\in n\Psi\}$, \eqref{e:23.2} follows as well. 

Finally, we observe  that for any events $B_1$ and $B_2$,
$B_1^c\triangle B_2^c=B_1\triangle B_2$, which in combination with
\eqref{e:almost.same}, makes the proof of \eqref{e:terms.13} for $i=1$
nearly identical to the proof of \eqref{e:23.1}, and the proof of
\eqref{e:terms.13} for $i=2$ nearly identical to the proof of
\eqref{e:23.2}, except that in the latter case we use Corollary
\ref{c:range}.



   

This completes the proof of \autoref{thm:cluster_size}. 
\qed

\begin{appendix}\label{appn}
\section*{Proof of Lemma \ref{lemma:WLLN_bounded}}
Since $E[Z_i]=0$, and $(\beta_{in})$ are uniformly bounded, 
we have $|E[Z_i\one(|\beta_{in}Z_i|\leq
\tau_n)]|\leq \delta/(4L)$ for all $n$ large enough and all $i$.  Together with $\sum_i \beta_{in} \leq nL$, this implies that
for all $n$ large enough, 
 \begin{align*}
     \left|E\left[\sum_i \beta_{in} Z_i\one(|\beta_{in}Z_i|\leq
\tau_n)\right]\right| \leq \sum_{i }|\beta_{in}|| E[Z_i\one(|\beta_{in}Z_i|\leq
\tau_n)] |
      \leq n\delta/2.
\end{align*}
For such $n$, 
\begin{align*}
      & P\left( \left| \sum_{i=-\infty}^\infty \beta_{in} Z_i
        \one(|\beta_{in}Z_i|\leq \tau n)\right| > n\delta\right) \\
      \leq&   P\left( \left| \sum_{i=-\infty}^\infty \beta_{in} \Bigl(
            Z_i\one(|\beta_{in}Z_i|\leq \tau n) - E[Z_i\one(|\beta_{in}Z_i|\leq \tau n)]\Bigr)
\right| > n\delta/2\right). 
\end{align*}
We use a concentration inequality, due to \cite{prokhorov:1959}, which
we state here for the sake of completeness. 
\begin{lemma} \label{lemma:Concentration}  
   Let $Y_1,\ldots, Y_m$ be  independent zero mean random variables
   such that for some $c>0$ we have 
   $Pr(|Y_i|\leq c) = 1$ for all $i$. Then for any $t> 0$, 
   $$P(Y_1+ \cdots + Y_m>t)\leq \left(\frac{ct}{\sum_{i=1}^m \Var(Y_i)}\right)^{-t/2c}.$$
\end{lemma}

If the series $\sum_{i} Y_i$ converges, then the statement of Lemma
\ref{lemma:Concentration} extends to the case $m=\infty$. We use the
concentration inequality in this situation, with  $Y_i =  \beta_{in} \Bigl(
            Z_i\one(|\beta_{in}Z_i|\leq \tau n) -
            E[Z_i\one(|\beta_{in}Z_i|\leq \tau n)]\Bigr) $,  so that
            $c = 2n\tau$. Here  $t =n\delta/2$.  We estimate the
variances using  \autoref{lemma:Varbound}. In the case $\alpha>2$ 
we have
\begin{align*}
&P\left( \left| \sum_{i=-\infty}^\infty \beta_{in} \Bigl(
            Z_i\one(|\beta_{in}Z_i|\leq \tau n) - E[Z_i\one(|\beta_{in}Z_i|\leq \tau n)]\Bigr)
\right| > n\delta/2\right)\\
\leq&  \left( \frac{n^2\tau\delta}{EZ_0^2 \sum_{i=-\infty}^\infty
    \beta_{in}^2}\right)^{-\delta/(8\tau)}
\end{align*}
for all $n$, 
proving the statement of part (a) of the lemma in this case.

In the case $1<\alpha\leq 2$, let $g/2$ be the regularly varying
function with index $2-\alpha$ from Lemma \ref{lemma:Varbound}. Then
by the Potter bounds (see e.g. Proposition 0.8 in \cite{resnick:1987}),
for every $0<\vep\leq 1$ there is $n_0=n_0(\vep, \tau/B)$ such that for all $n\geq
n_0$, 
\begin{align*}
  &\Var(Y_i) = \Var\bigl( \beta_{in} Z_i\one(|\beta_{in}Z_i|\leq \tau n)\bigr)
  = \beta_{in}^2 g\bigl( \tau n/|\beta_{in}|\bigr)/2 \\
\leq&  \beta_{in}^2 (1+\vep)(B/|\beta_{in}|)^{2-\alpha+\vep} g\bigl( \tau
      n/B\bigr)/2
      \leq B^{2-\alpha+\vep} |\beta_{in}|^{\alpha-\vep}g\bigl( \tau
      n/B\bigr),
\end{align*}
and he statement of part (a) of the lemma in this case once again
follows from the concentration inequality. 

For part (b) we write
\begin{align*}
 P\left( \left| \sum_{i=-\infty}^\infty \beta_{in} Z_i
       \right| > n\delta\right) \leq& P\bigl( \max_i |\beta_{in}Z_i| >
  n\tau \bigr) \\
+& P\left( \left| \sum_{i=-\infty}^\infty \beta_{in} Z_i
        \one(|\beta_{in}Z_i|\leq \tau n)\right| > n\delta\right).
\end{align*}
The required bound on the second term in the right-hand side above is
obtained in part (a), so the proof is completed by noticing that by
the Potter bounds, for every $0<\vep\leq 1$ there is $n_0=n_0(\vep,
\tau/B)$ such that for all $n\geq n_0$, 
\begin{align*}
    P(\max_i |\beta_{in} Z_i |> n\tau)\leq& \sum_{i} P(|\beta_{in} Z_i |> n\tau) \\
\leq& (1+\vep) B^{-\alpha +\vep} P(|Z_0|>n\tau/B) \sum_{i}
     |\beta_{in}|^{\alpha-\vep}. 
   \end{align*}
\qed

\end{appendix}

\begin{funding}
This work was partially supported by NSF grant DMS-2310974 at Cornell University.
\end{funding}


\end{document}